\documentclass[12pt]{article}
\usepackage[pdftex]{graphicx} 
\usepackage{subfig}
\usepackage[utf8]{inputenc}
\usepackage{hyperref}
\usepackage{amsmath}
\usepackage{amsthm}
\usepackage{amsfonts}

\newtheorem{theorem}{Theorem}
\newtheorem{prop}{Proposition}

\DeclareMathOperator{\Res}{Resultant}
\newcommand{\tr}[1]{\mathop{{\rm tr} #1}}

\title{Bautin bifurcation in a minimal model of immunoediting}
\author{
Joaquín Delgado\thanks{Departamento de Matemáticas, UAM-Iztapalapa Av. San Rafael Atlixco 186, Col. Vicentina, C.P. 09340 CDMX. Email: \href{mailto:jdf@xanum.uam.mx}{jdf@xanum.uam.mx}} 
\and Eymard Hern\'andez\thanks{Posgrado en Ciencias Naturales e Ingeniería, UAM-Cuajimalpa. Av. Vasco de Quiroga 4871, Col. Santa Fé Cuajimalpa, C.P. 05348, CDMX. Email: \href{mailto:eymardh7@gmail.com}{eymardh7@gmail.com} }
\and 
Luc\'{\i}a Ivonne Hern\'andez-Mart\'{\i}nez\thanks{Universidad Autónoma de la Ciudad de México, Plantel San Lorenzo Tezonco. Calle Prolongación San Isidro 151, Col. San Lorenzo Tezonco, C.P. 09790 CDMX. Email: \href{mailto:eivhernandezster@gmail.com}{ivhernandezster@gmail.com } }
}

\begin{document}
 
\maketitle
\begin{abstract} One of the simplest model of 
  immune surveillance and neaoplasia was proposed by Delisi and Resigno~\cite{Delisi}. Later Liu et al~\cite{Dan} proved the existence of  non-degenerate Takens-Bogdanov bifurcations defining a surface in  the whole set of five positive parameters. In this paper we prove the existence of  Bautin bifurcations completing the  scenario of possible codimension two bifurcations  that occur in this model. We give an interpretation of our results in terms of the three phases   immunoediting theory:elimination, equilibrium and escape.
\end{abstract}
\noindent\textbf{Key words}: Bautin bifurcation, cancer modeling, immunoediting.

\noindent\textbf{2000 AMS classification:} Primary: 34C23, 34C60; Secondary: 37G15.

\section{Introduction}
Immune edition conceptualices the development of cancer in three phases \cite{Kim}. In the first one, formerly known as immune surveillance, the complex of the immune system eliminates cancer cells originating from an intrinsic fail in the supresor mechanisms. When some part of  cancer cells are eliminated an equilibrium between the immune system and the population of cancer cells is achieved, leading to a durming state. Then the cancer cells accumulate genetic and epigenetic alterations in the DNA that generate specific stress-induced antigens. When a disbalance of the cancer polulation occurs the explosive phase appear with a fast growth of tumor cells. One of the simplest models in the first stage of the immune edition framework, based on a previous model of Bell \cite{Bell}, is due to Delisi and Resigno \cite{Delisi}. They model  the population of cancer cells and lymphocites  as a predator--prey system. The cancer tumor grows in the early stage as a spherical tumor that protects the inner cancer cells. Only the cancer cells on the surface of the tumor interact with the lymphocites. Under proper hypotheses on the balance of the total cancer cells and allometric growth, they propose a model of two ODEs depending on five parameters. 

Years after, Liu, Ruan and Zhu \cite{Dan}, study the nonvascularized model of \cite{Delisi} and prove that a Takens-Bogdanov bifurcation of codimension two occurs. 

The nonvascularized model of Delisi is
\begin{equation}\label{Delisi0}
\begin{array}{lcll}
\frac{dx}{dt} &=& -\lambda_1 x + \frac{\alpha_1 x y^{2/3}}{1+x}\left(1-\frac{x}{x_c}\right),&\\
\frac{dy}{dt} &=& \lambda_2 y -\frac{\alpha_2 x y^{2/3}}{1+x}
\end{array}
\end{equation}
where $x$ is the number of free lymphocites that are not bounded to cancer cells, $y$ is the total number of cancer cells in adimensional variables. The fractional power is the result of assuming an allometric law of the number of cancer cells on the surface of an spherical tumor. Obviously the model is not well suited for $y=0$ which correspond to the initial tumor cell being a point. In fact the theorem of uniqueness of solutions does not hold for inital conditions of the form $(x_0,0)$.
After a change of variables $\bar{x}=x$, $\bar{y}=y^{1/3}$, perform the next reparametrization 

\begin{equation}\label{Repara}
   \frac{dt}{d\bar{t}}= 1+x,
\end{equation}
and droping the bars the system becomes the polynomial system

\begin{equation}\label{Delisi}
\begin{array}{lcll}
      \frac{dx}{dt}&=&-\lambda_{1} x (1+x) +\alpha_{1} \left(1-\frac{x}{x_{c}}\right)x y^{2}&\\
    \frac{dy}{dt}&=&\lambda_{2}(1+x)y-\alpha_{2}x,
\end{array}
\end{equation}
Consider $(x_{0},y_{0})$ critical point of the system, then 
\begin{eqnarray}
y_0 &=&\frac{\alpha_2x_0}{\lambda_2(1+x_0)}\label{y0}\\
\frac{\lambda_1\lambda_2^2}{\alpha_1\alpha_2^2}&=&\frac{x_0^2(1-x_{0}/x_c)}{(1+x_0)^3}\label{x0}
\end{eqnarray}
Therefore the abscissa $x_0$ of the critical points are determined by the roots of the cubic polynomial  (\ref{x0}). In what follows the combination of parameters 
\begin{equation}
    \psi\equiv\frac{\lambda_1\lambda_2^2}{\alpha_1\alpha_2^2},\quad \lambda=\frac{\lambda_2}{\lambda_1}
\end{equation}
will be very useful.
In particular the critical points can be described by the catastrophe surface
\begin{equation}\label{cubic}
  \Sigma=\{(\psi,x_c,x_0)\mid   x_0^2(1-x_0/x_c)-\psi (1+x_0)^3=0\}.
\end{equation}
in the space of parameters $\psi$--$x_c$ and abscissa $x_0$. This surface is shown in Figure~1. The plane $x_0=0$ correspond to the trivial critical point $(0,0)$ and is a saddle. The red line shows a case of value of the parameters $(\psi,x_c)$ such that there are three critical points determined by  their $x_0$ abscissa. At a point where the surface folds back, the number of critical point is three, counting the trivial one. The projection of this folding is given by the discriminant of the cubic,
\begin{equation}\label{discriminant}
    \Delta= 4x_c^2-27 (1+x_c)^2 \psi =0,\quad\mbox{or}\quad
     \psi= \frac{4x_c^2}{27(1+x_c)^2},
\end{equation}
defines a curve in the parameter plane  $\psi$--$x_c$ where the projection $(\psi,x_c,x_0)\mapsto x_0$ restricted to $\Sigma$ looses range and the catastrophe surface folds back.

\begin{figure}[t]
    \centering
    \includegraphics[height=2.5in]
    {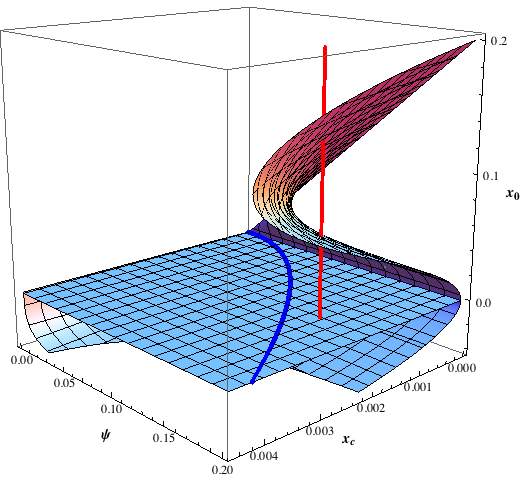}
    \caption{The catastrophe surface in coordidantes $(\psi, x_c, x)$, $x_0$ the abscisa of the critical point.  For a given value of  $(\psi,x_c)$ there are up to two critical points with $x_0>0$ and the trivial critical point corresponding to $x_0=0$. Notice that there are critical points with $x_0<0$ that are not considered. The foldding of the surface projects into the saddle--node curve given by (\ref{discriminant}) in the plane $\psi$--$x_c$.}
    \label{fig:surface}
\end{figure}

The rest of the paper is organized as follows: In section 2 we summarize the results of Liu et al regarding the existence of saddle--node and Takens--Bogdanov bifurcations. In section 3 we state the main result of this paper, the existence of Bautin bifurcations and describe it explicitly in terms of a proper parametrization. We give the main idea of the proof and the details are posponed to the Appendix A. The global bifurcation diagram is completed numerically with MatCont using the local diagrams of the Takens-Bogdanov and Bautin bifurcation as  described in the Appendix C. In section 4  we  describe the phase portraits derived from the global bifurcation diagram and represented schematically in Figure~\ref{schema}. Finally in Section 5 we give an interpretation of our results. In Appendix C we describe briefly how the numerical continuation with MatCont was performed.

\section{Saddle node and Hopf bifurcations}
The following two results sumarizes the results by Liu et al \cite{Dan}.
\begin{prop}[Liu et al]
The parameter set
$$
SN=\left\{(\lambda_1,\lambda_2,\alpha_1,\alpha_2,x_c)\mid \psi= \frac{4x_c^2}{27(1+x_c)^2},\, \lambda\neq\frac{2(3+x_c)}{3(1+x_c)} \right\}
$$
are saddle node bifurcations of system (\ref{Delisi}). The phase portrait consists of two hyperbolic and one parabolic sectors.
\end{prop}

Using (\ref{discriminant}) we can obtain the explicit parametrization of the saddle--node curve in the plane $\lambda_1$--$\lambda_2$ for given values of $\alpha_1$, $\alpha_2$ and $x_c$. 
$$
\lambda_1 = \frac{1}{3}\left(\frac{4x_c^2 \alpha_1 \alpha_2^2\lambda^2}{(1+x_c)^2} \right)^{1/3},\qquad
\lambda_2 = \frac{1}{3}\left(\frac{4x_c^2 \alpha_1 \alpha_2^2}{\lambda (1+x_c)^2} \right)^{1/3}
$$

Takens-Bogdanov bifurcations are given as follows:

\begin{theorem}[Liu et al]
The parameter set
\begin{equation}
BT=\left\{(\lambda_1,\lambda_2,\alpha_1,\alpha_2,x_c)\mid \psi= \frac{4x_c^2}{27(1+x_c)^2},\, \lambda=\frac{2(3+x_c)}{3(1+x_c)}\right\}
\end{equation}
are non-degenerate Takens-Bogdanov  bifurcations of system (\ref{Delisi}). 
\end{theorem}
For a choice of parameters in $BT$ the critical point undergoing a BT bifurcation is given by (\ref{y0}) and (\ref{x0}). As a previous construction towards proving our main result, we first characterize the Hopf bifurcations locus.

\begin{prop}\label{1}
The parameter set
$$
H=\{(\lambda_1,\lambda_2,\alpha_1,\alpha_2,x_c)\mid \mbox{(\ref{Hopf}) holds}\}
$$
is the Hopf and symmetric saddle bifurcation surface of system (\ref{Delisi}).
\begin{eqnarray}
0 &=&
(1+x_c)^3\psi\lambda^3 - (\psi x_c^3 +(1-\psi)x_c^2 -5\psi x_c - 3\psi)\lambda^2 +\nonumber\\
&& (x_c^2+4x_c+3)\psi\lambda +(1+x_c)^2(1+x_c\psi)\psi
   \label{Hopf}
\end{eqnarray}
\end{prop}
\begin{proof}
Let $f$, $g$ denote the right hand sides in (\ref{Delisi}), then we look for a common root of the polynomial equations $f=g=trA=0$, where
 $A=\frac{\partial(f,g)}{\partial(x,y)}$ and $trA=\tr{A}$. We compute
$R_1=\Res[trA,f,y_0]$, 
$R_2=\Res[trA,g,y_0]$ which are polynomials in $x_0$. A necessary condition for $trA=0=f$ to have a common root is that $R_1=0$, and similarly a necessary condition for $trA=g=0$ to have a common root is that $R_2=0$. Then compute $R=\Res[R_1,R_2,x_0]$ which is a polynomial in the parameters. A necessary condition for $R_1=R_2=0$ to have a common root is that $R=0$.  If  we exclude trivial factors, we end up with (\ref{Hopf}).
\end{proof}

Liu et al \cite{Dan} prove that a non--degenerate Takens-Bogdanov bifurcation occurs for any values of the positive parameters, thus excluding the possibility of codimension three degeneracy.
Adam  \cite{Adam} gives sufficient conditions for  system (\ref{Delisi}) to undergo a Hopf bifurcation, although no explicit computation is done. Liu et al describe the Hopf bifurcation locus in terms of parameters involved in the normal form computation, thus not explicit. The expression in Proposition~\ref{1} gives an explicit parametrization of the locus of Hopf bifurcations in the parameters.

\section{Bautin bifurcation}
We now give the main idea to compute the first Lyapunov coefficient for a critical point undergoing a Hopf bifurcation. Let $(x_0,y_0)$ be such a critical point. Then we shift the critical point to the origin $x=x_0+\epsilon x_1$, $y=y_0+\epsilon y_1$ and expand in powers of $\epsilon$ in order to collect the homogenous components of the vector field. We first consider the linear part
\begin{eqnarray}
x_1' &=& ax_1+b x_1,\\
y_1' &=& cx_1+d x_1
\end{eqnarray}
 and perform the linear change of variables $Y_1=b_2 x_1-a_2 y_1$, $Y_2=(a_1 b_2-a_2 b_1)x_1$. Under the hypothesis of complex eigenvalues and  the determinant $a_1b_2-a_2 b_1>0$ the system reduces to an oscillator equation $Y_1'=Y_2$, $Y_2'=\omega^2 Y_1-2\mu Y_1$, with eigenvalues $\lambda=\mu\pm\sqrt{\omega^2-\mu^2}$ and the Hopf condition becomes $\mu=0$, $\omega^2= a_1b_2-a_2b_1$. We compute right and left eigenvectors $q_0$, $p_0$ such that $Aq_0=i\omega q_0$ and $A^{T}p_0=-i\omega p_0$ and $\langle p_0,q_0\rangle=1$. Then $\langle p_0,\bar{q_0}\rangle=0$. Let $Y=z q_0+\bar{z}\bar{q_0}$. Then the whole nonlinear system reduces to (setting $\epsilon=1$)
 $z'=\lambda z+ G_2(z,\bar{z})+G_3(z,\bar{z})+\cdots$
 then we compute $\ell_1$ by the formula given by \cite[p.309--310]{Yuri}.
 
 As shown in the appendix, $\ell_1$ becomes a polynomial in $x_0,y_0,\omega$ and after elimination of $\omega^2$ and $\omega^4$ which are the only powers appearing there, and  of $y_0$ using (\ref{Delisi}), a polynomial in $x_0$ of high degree (19) results. The main difficulty is that computing the abscissa $x_0$ of the critical point amounts to solving a cubic polinomial. Therefore we compute the resultant of $\ell_1$ with the cubic polinomial (\ref{cubic}) and eliminate $x_0$. Taking an appropriate factor of this, we then compute its resultant with the Hopf equation (\ref{Hopf}). There are two factors. One of this leads to the solution for $\lambda=\lambda_2/\lambda_1$,
 $$
 \lambda=\frac{-3+x_c}{3(1+x_c)}
$$
Substituting this value in the Hopf equation (\ref{Hopf}) we solve for $\psi$ in an appropriate factor. We then get the following

\begin{theorem}
The  parameter set
\begin{eqnarray}
Bau&=&\left\{(\lambda_1,\lambda_2,\alpha_1,\alpha_2,x_c)\mid \,\psi=\frac{\sqrt{x_c}\left(
(-27+x_c)\sqrt{x_c}+ (9+x_c)^{3/2}
\right)}{27(1+x_c)^2},\right.\nonumber\\
&& \left.\qquad\lambda=\frac{-3+x_c}{3(1+x_c)}\right\}
\end{eqnarray}
are Bautin points of codimension $2$ of system (\ref{Delisi}). 
\end{theorem}

\subsection{Bifurcation diagram around a point of Bautin}
The local bifurcation diagram around a Bautin point is shown in Figure \ref{Bautin-fig} (see \cite[p.313]{Yuri})
\begin{figure}
    \centering
    \includegraphics[scale=0.3]{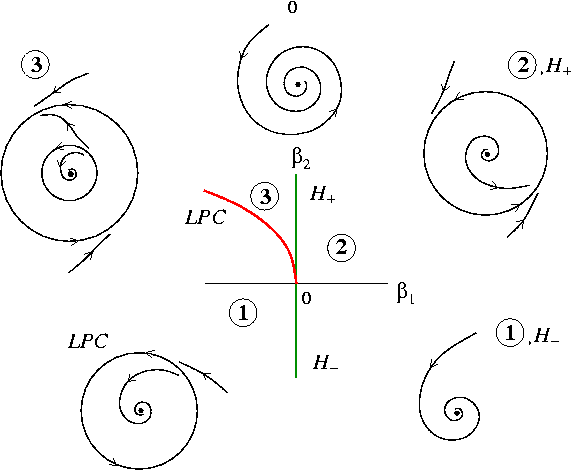}
    \caption{Local diagram of Bautin bifurcacion}
    \label{Bautin-fig}
\end{figure}
There are two components of the Hopf curve $H_{\pm}$ correspondig to the sign of the first Lyapunov coefficient $\ell_0$. Thus when crossing the component $H_{-}$ from positive values of $\beta_1$ a stable limit cycle appears, and smilarly, when crossing the component $H_{+}$, an unstable limit cycle appears. Therefore in the cusp region 3, there coexist two limit cycles the exterior on being stable, the interior unstable, and both collapse along the LPC curve.

\section{Global dynamics}
Figure \ref{schema}  shows schematically the bifurcation diagram  as computed numerically with MatCont in Figure~\ref{fig:diagram}. There are shown three lines of fixed value of $\lambda_2$ varying $\lambda_1$. We will now describe the qualitative phase portrait along these lines. For the upper line $CT$ corresponding to a value of $\lambda_2$ just below the Takens--Bogdanov point $BT$, the dynamics can be described as follows: In passing from a point $C$ to a point $D$ the trivial critical point connects to the saddle point along a hetheroclinic orbit. This happens at the point marked as $K$. Indeed a curve of heteroclinic connections is depicted along the points $KK'K''$ although we have not computed it numerically.  The transition from $C$ to $D$ passing through the heteroclinic connection $K$,  and further evolution to a limit cycle bifurcating from a homoclinic connection at $P$, and disappearance of the limit cycle through a transcritical Hopf bifurcation  ending at $T$,  is shown in Figure~\ref{LineCT}. For completeness we have included the flow at infinity as described in Appendix~\ref{Blow-up}. The critical points at infinity $y=\infty$ are shown as blue points. Notice the hyperbolic sector for $x=0$ and the attractor at $x=x_c$.

\begin{figure}
\centering
\includegraphics[scale=0.75]{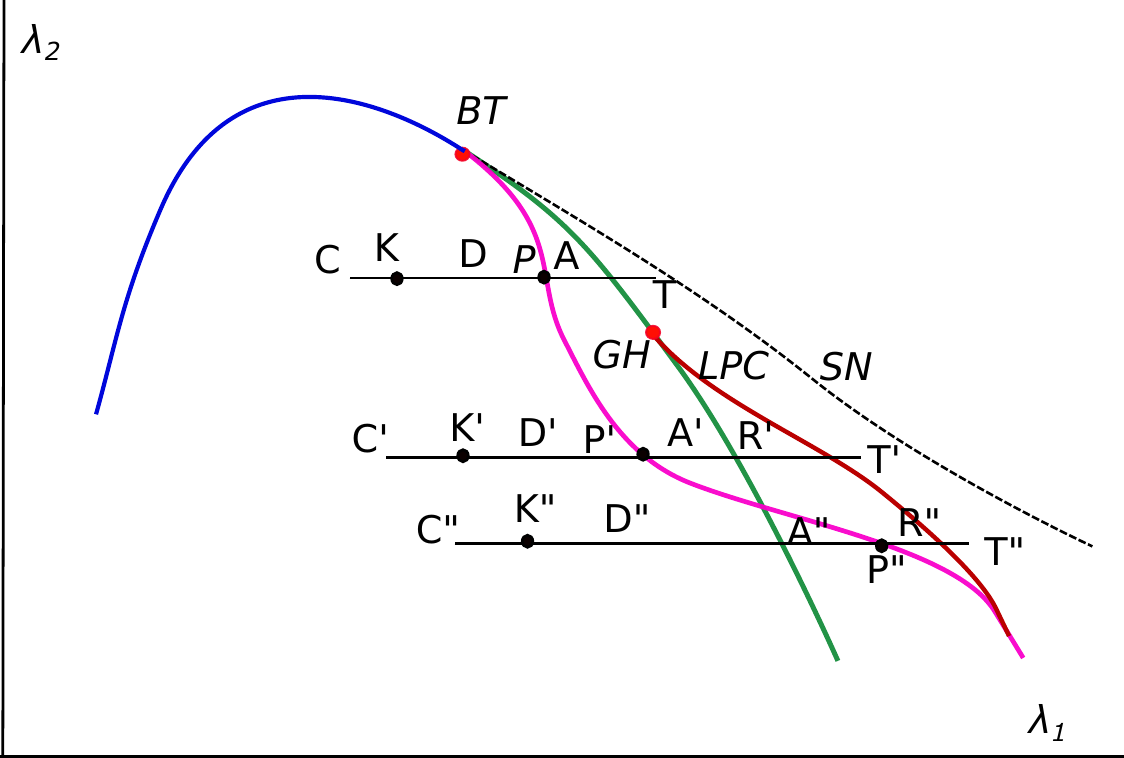}
\caption{Schema of the bifurcation diagram \ref{Bautin-fig}.\label{schema}}
\end{figure}

\begin{figure}\centering
\includegraphics[scale=0.35]{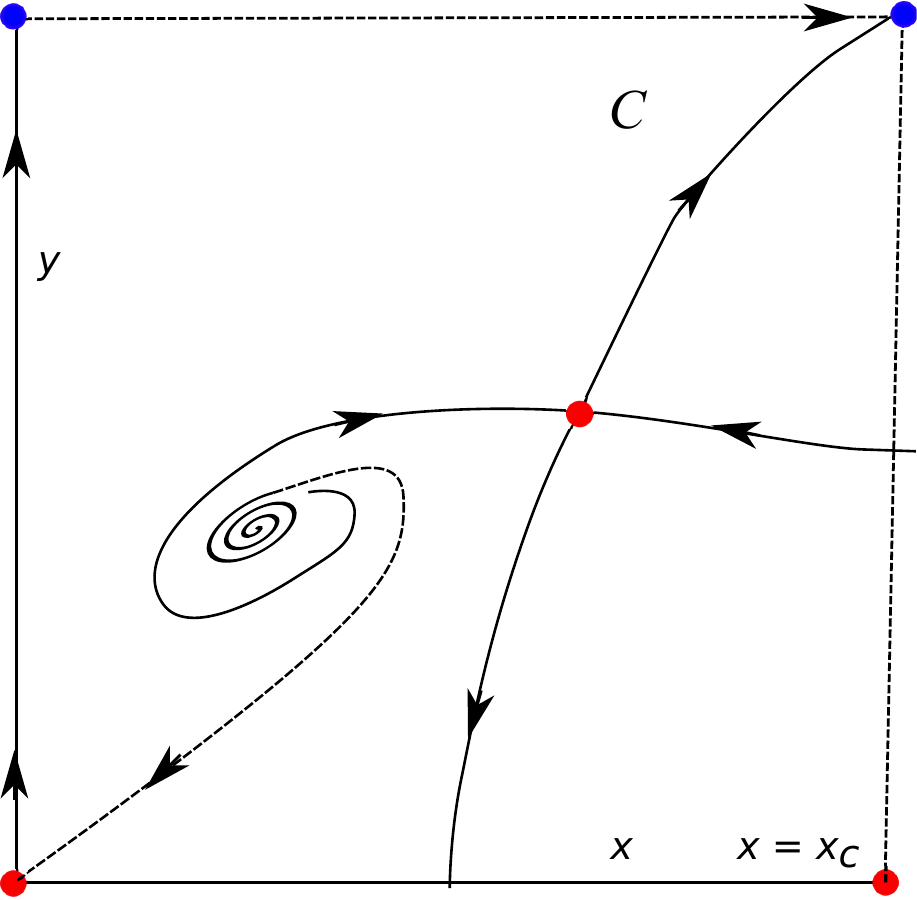}
\includegraphics[scale=0.35]{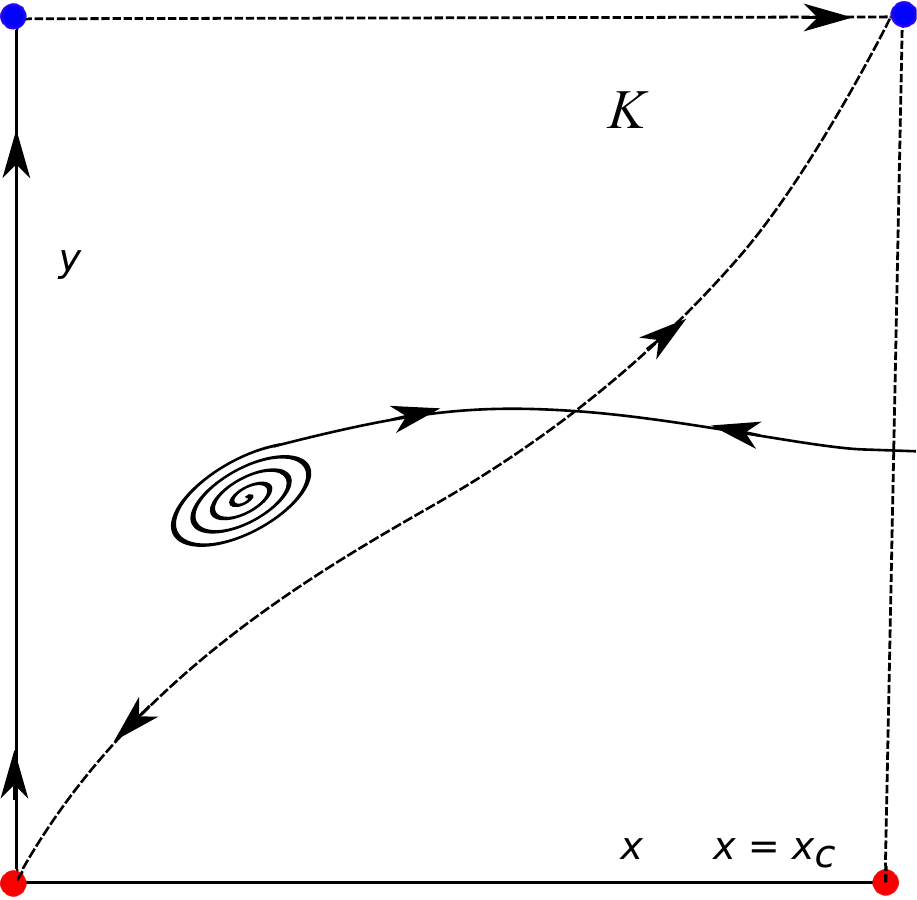}
\includegraphics[scale=0.35]{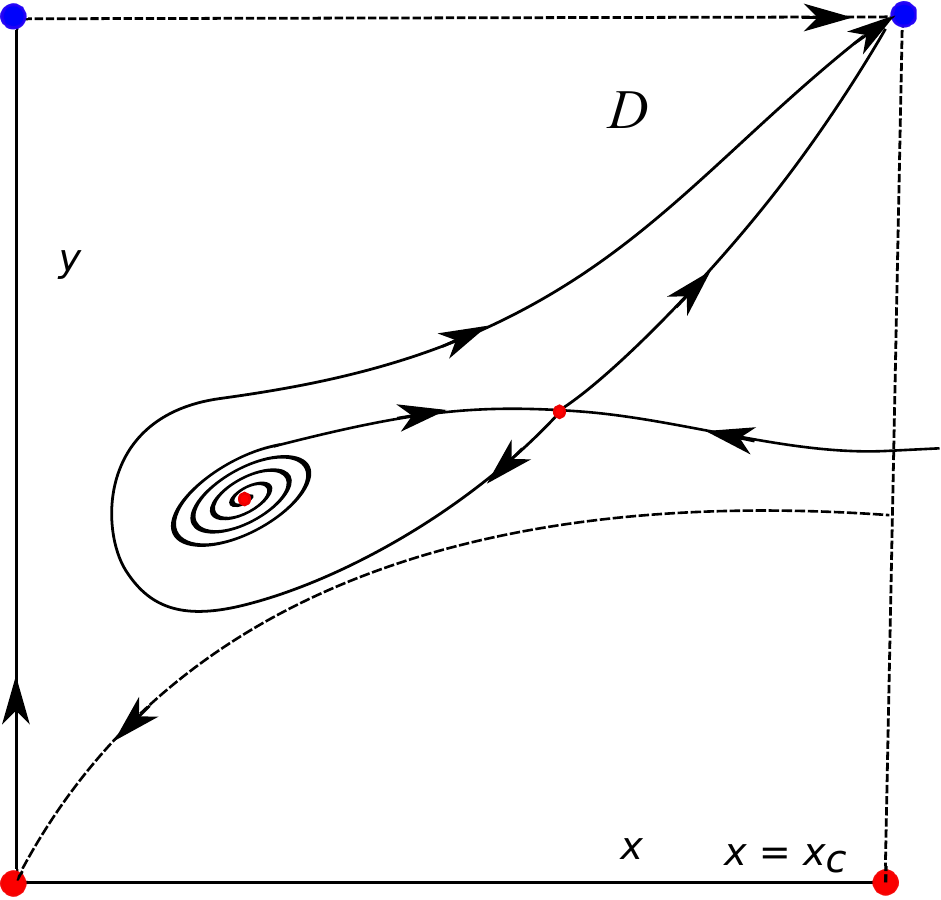}\\
\includegraphics[scale=0.35]{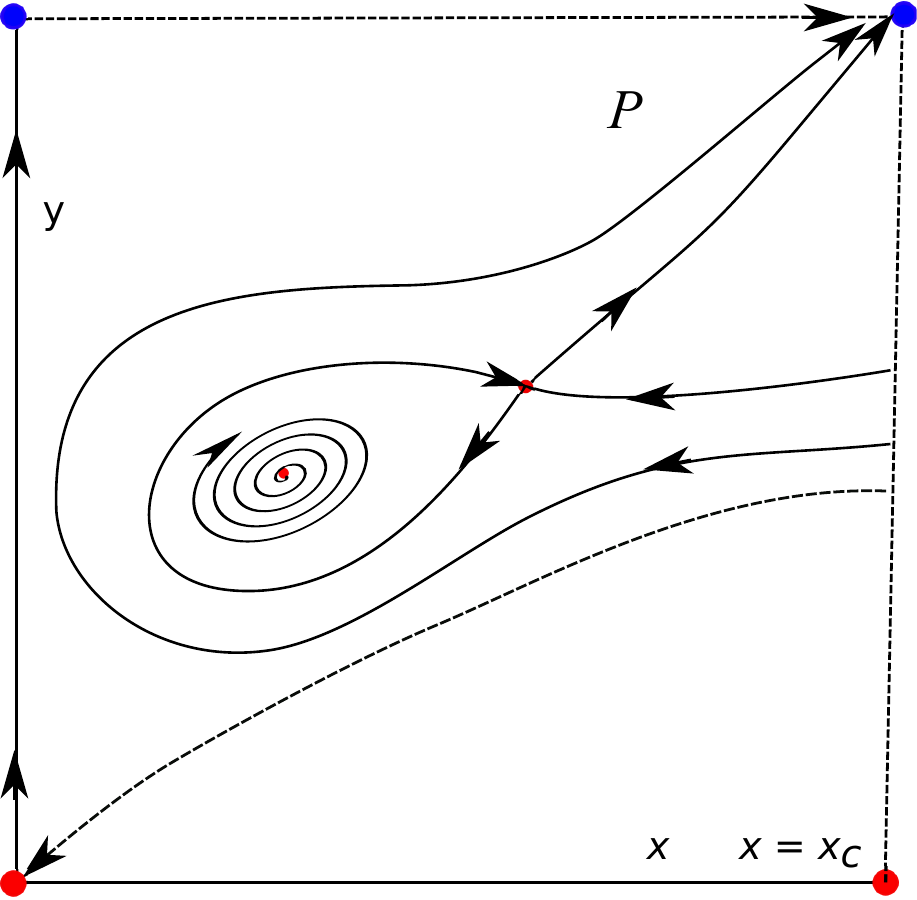}
\includegraphics[scale=0.35]{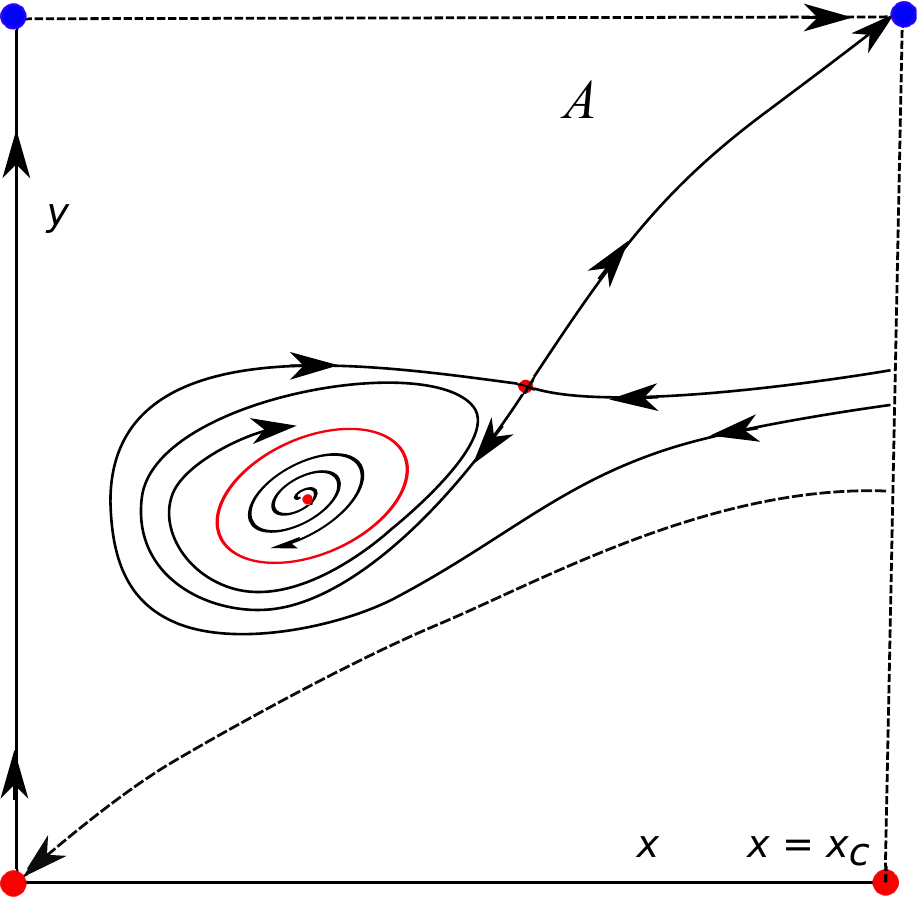}
\includegraphics[scale=0.35]{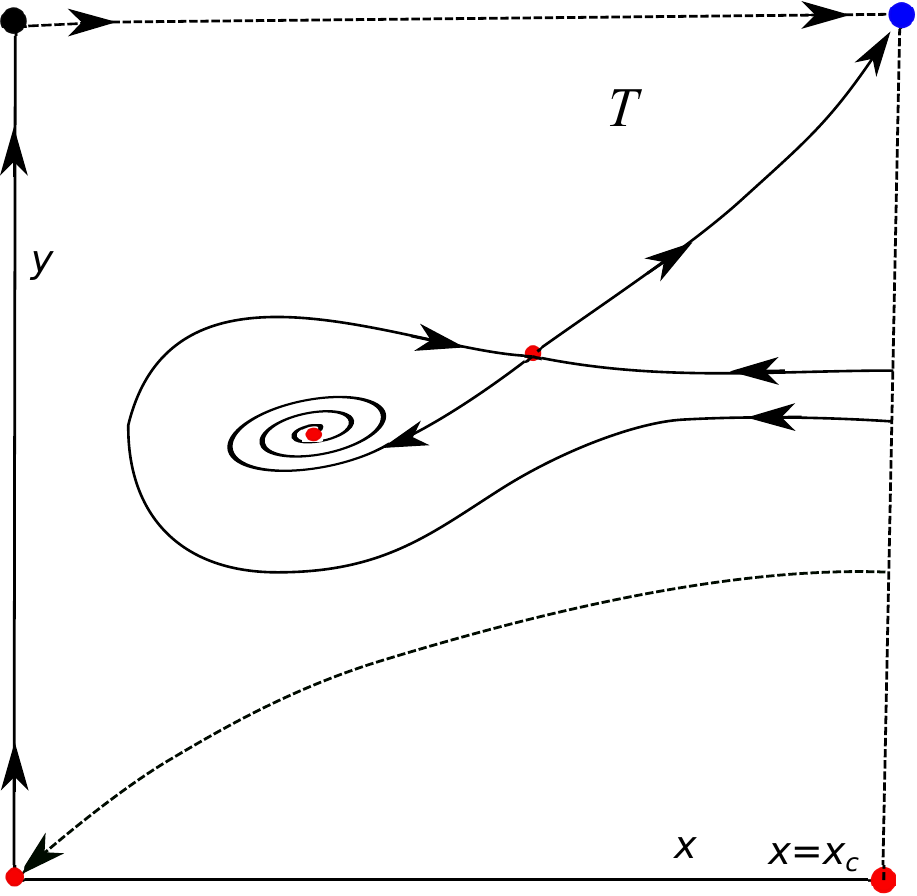}
\caption{Qualitative phase portrait  along the line $CKDPAT$ of the bifurcation scheme in Figure \ref{schema}.\label{LineCT}}
\end{figure}

Similarly, the evolution of the phase portrait along the line $C'T'$ is described in Figure~\ref{LineCprimeTprime}. The evolution along the part $C'K'D'P'A'$ is the same as $CKDPA$ in Figure~\ref{LineCT}, the difference is at the further development of an unstable limit cycle inside the stable limit cycle previously created by a homoclinic bifurcation at $P'=P$, as shown in figure $R'$ and further desappearence of both limit cycle as in $T'$ through a limit point of cycles.

\begin{figure}\centering
\includegraphics[scale=0.35]{C}
\includegraphics[scale=0.35]{K}
\includegraphics[scale=0.35]{D}
\includegraphics[scale=0.35]{P}\\
\includegraphics[scale=0.35]{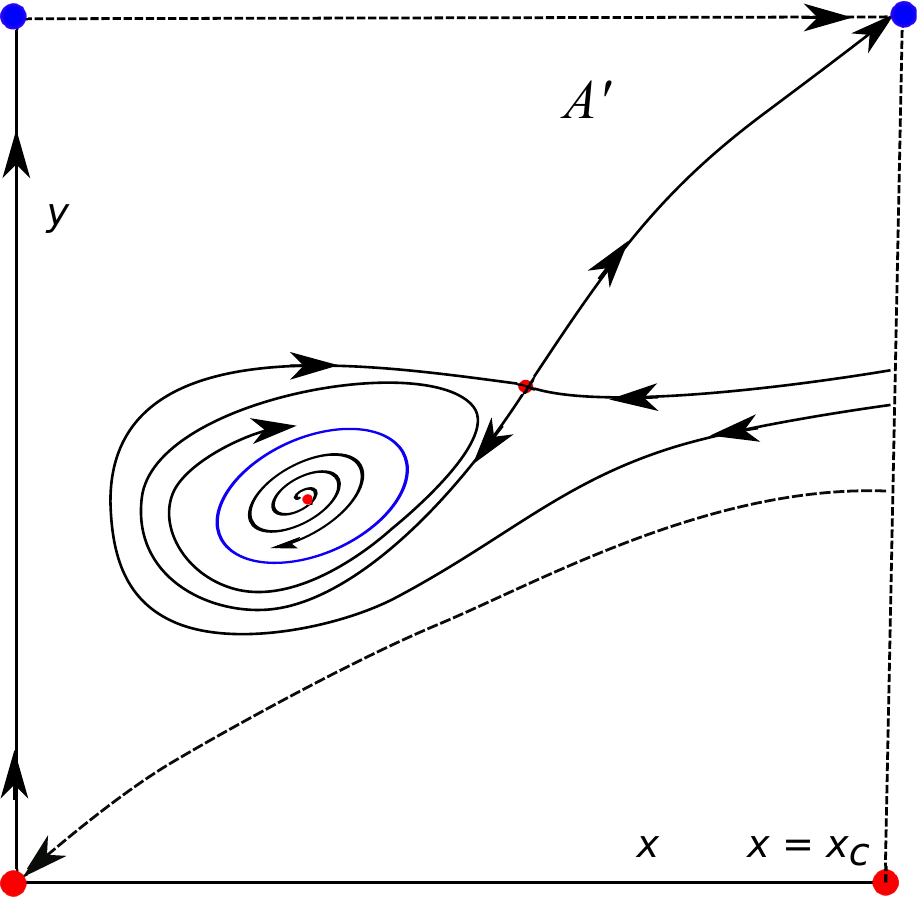}
\includegraphics[scale=0.35]{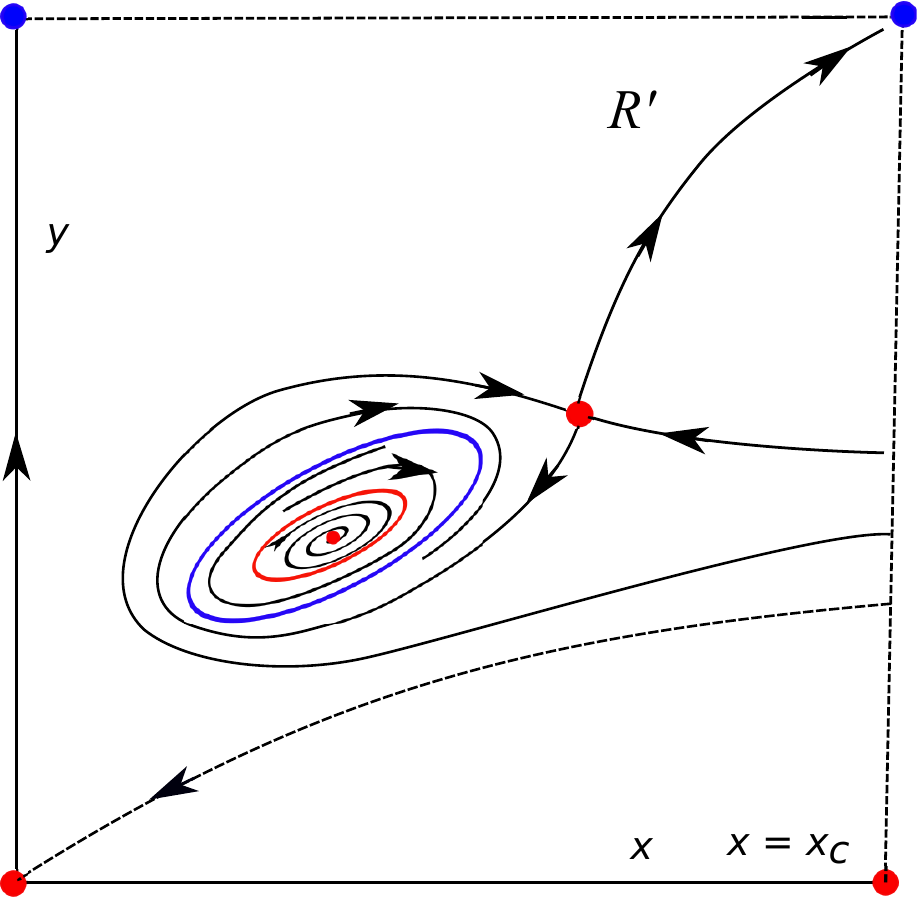}
\includegraphics[scale=0.35]{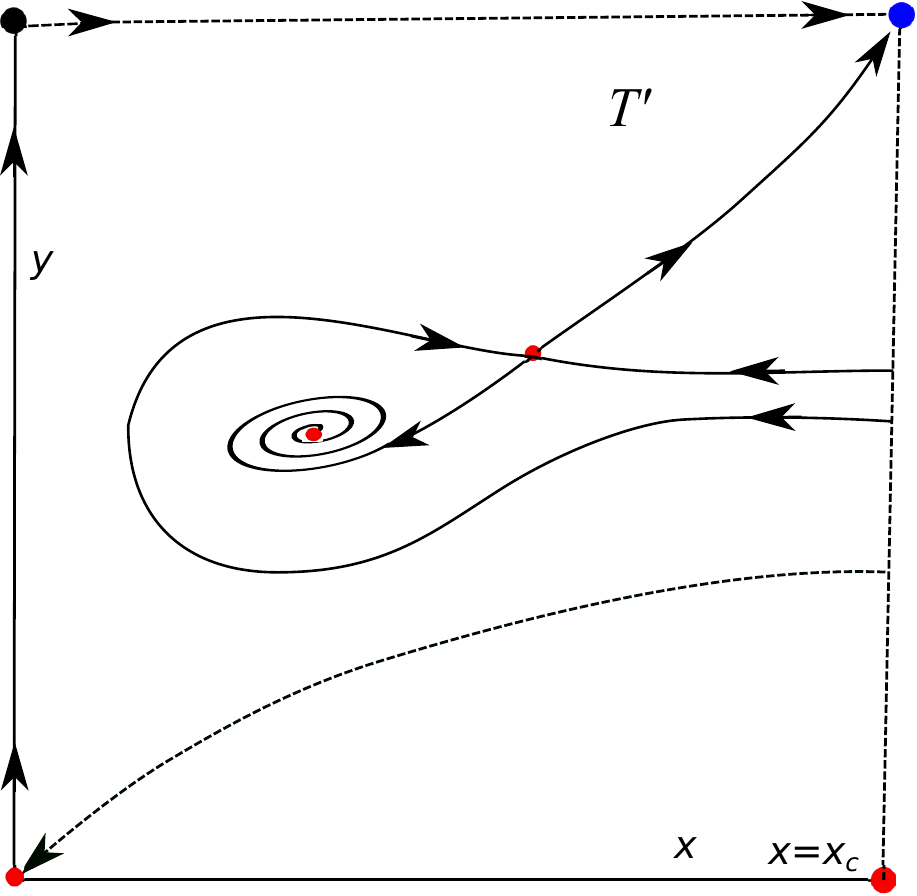}
\caption{Qualitative phase portrait  along the line $C'K'D'P'A'T'$ of the bifurcation scheme in Figure \ref{schema}. The phase portrait along the segment $C'K'D'P'A'$ is the same as $CKPDA$. \label{LineCprimeTprime}}
\end{figure}

Finally the evolution along the line $C''T''$ is described as follows: The phase portrait along  $C''K''D''$ is the same as in $C'K'D'$. Differently from the previous case,  after $D''$ a Hopf bifurcation occurs and an unstable limit cycle is appear as in $A''$ and then a second stable limit cycle originating in an homoclinic bifurcation leading to coexistence of two limit cycles as in case~$R'$.  The whole evolution along the line $C''T''$ is shown in Figure~\ref{LineCbiprimeTbiprime} where only the phase portraits different from the previous case are denoted as $A''$ and $P''$,

Figure~\ref{Coexistence}-(a), (b) shows in detail the evolution along $C'T'$ in the triangular region of coexistence of two limit cycles, with $\lambda_1$ as the $z$-axis. Notice that along increasing values of $\lambda_1$, first a limit cycle bifurcates from a homoclinic an then the second cycle bifurcates from a Hopf point. Figure ~\ref{Coexistence}-(c), (d)  evolution along $C''T''.$

\begin{figure}\centering
\includegraphics[scale=0.35]{C}
\includegraphics[scale=0.35]{K}
\includegraphics[scale=0.35]{D}
\includegraphics[scale=0.35]{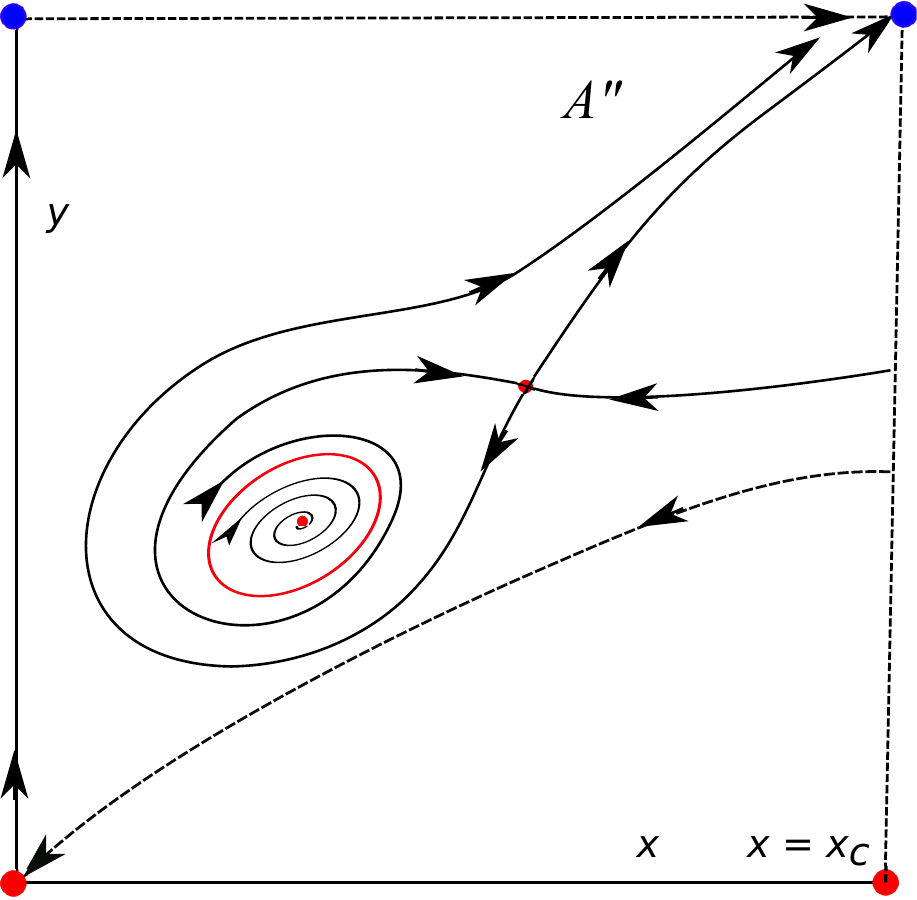}\\
\includegraphics[scale=0.35]{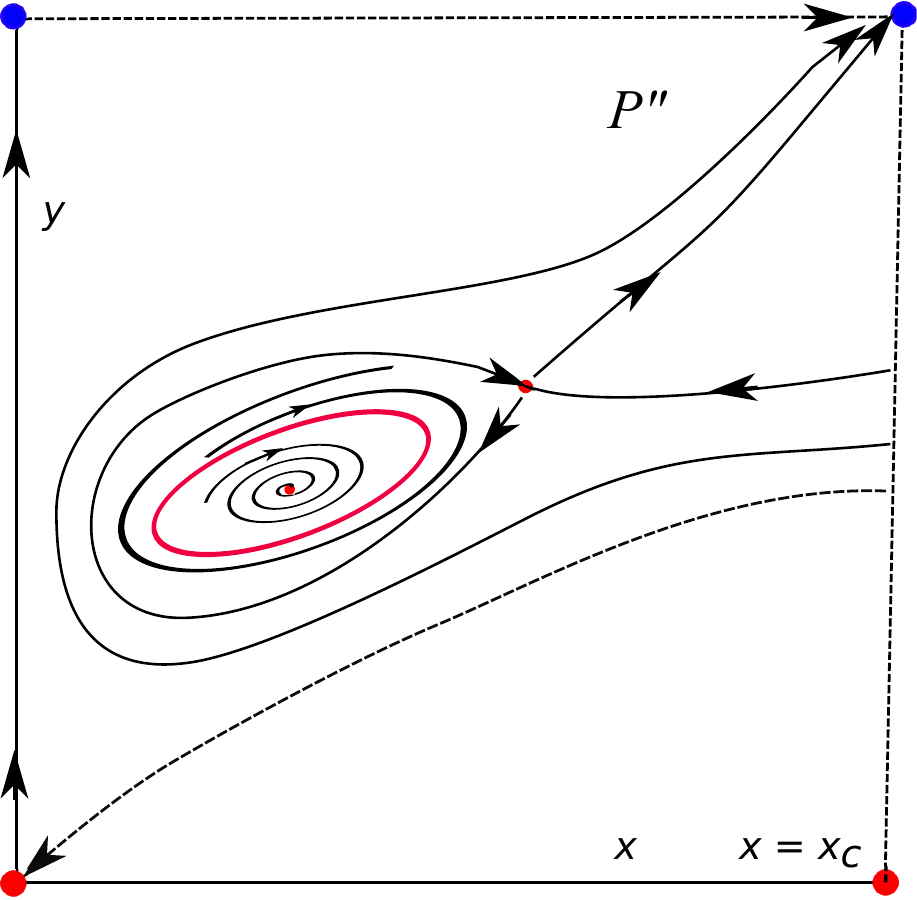}
\includegraphics[scale=0.35]{Rp}
\includegraphics[scale=0.35]{Tp}
\caption{Qualitative phase portrait  along the line $C''K''D''P''A''T''$ of the bifurcation scheme in Figure \ref{schema}.\label{LineCbiprimeTbiprime}}
\end{figure}

\begin{figure}
    \centering
    \subfloat[Numerical continuation of limit cycles in coordinates $(x,y,\lambda_1)$ along the line $C'T'$ within the region of coexistence of two limit cycles.]{{
    \includegraphics[scale=0.3]{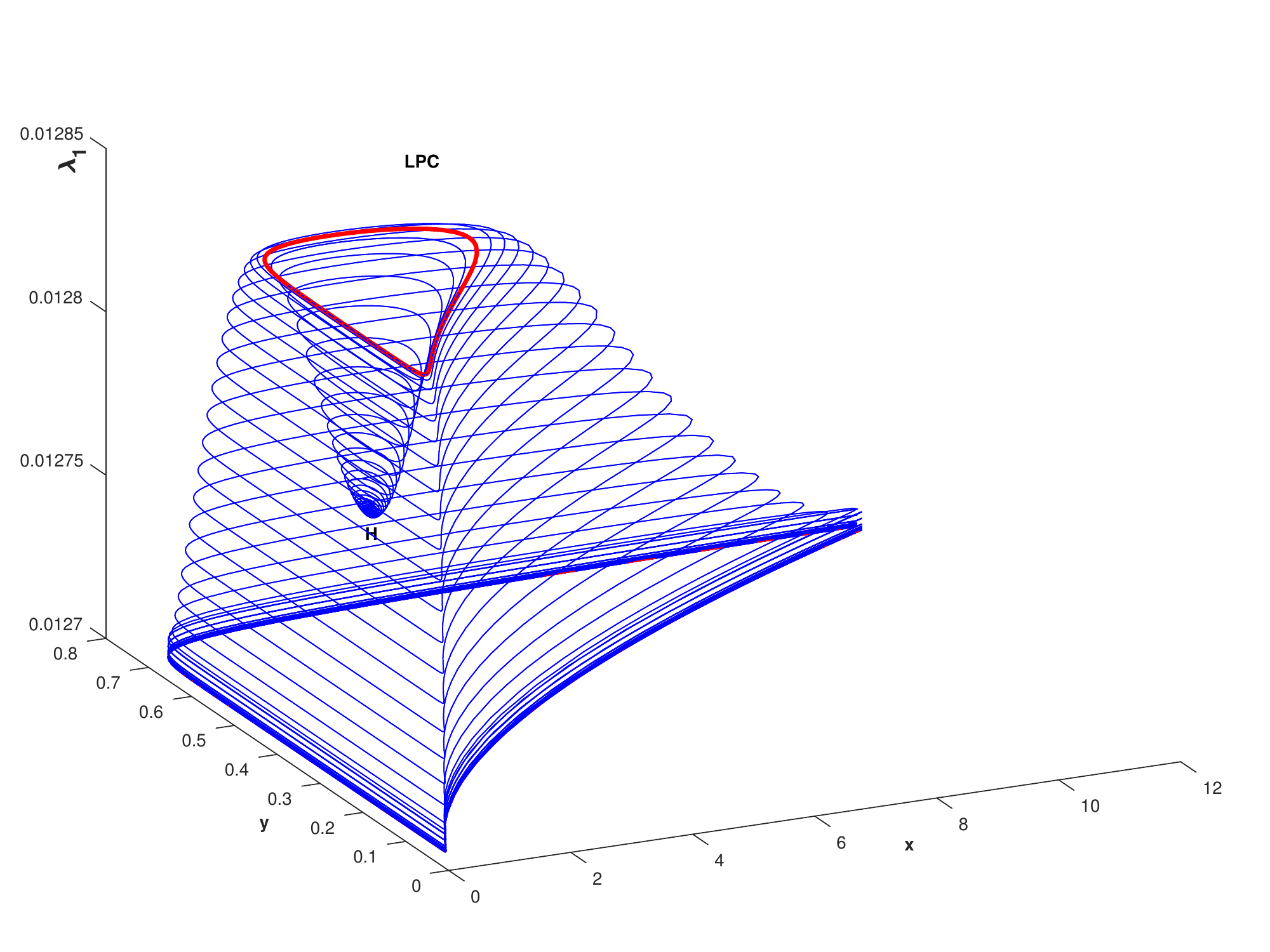} 
    }}
    \qquad
    \subfloat[Numerical continuation of limit cycles in coordinates $(x,y,\lambda_1)$ along the line $C''T''$ within the region of coexistence of two limit cycles.]{{
    \includegraphics[scale=0.3]{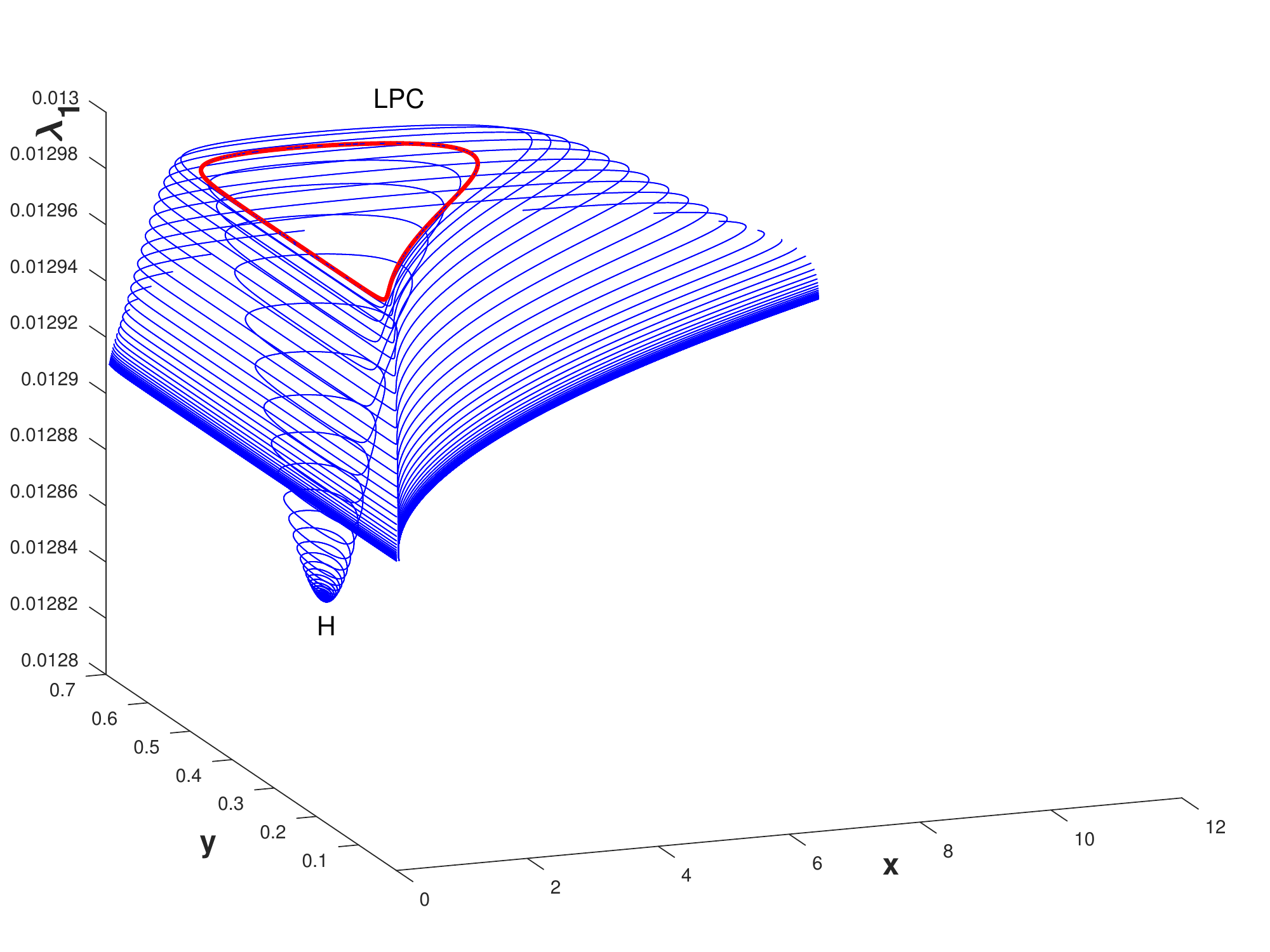} 
    }}\\
    
    \subfloat[Parameter $\lambda_1$ versus period $T$ of the cycles.]{{
    \includegraphics[scale=0.75]{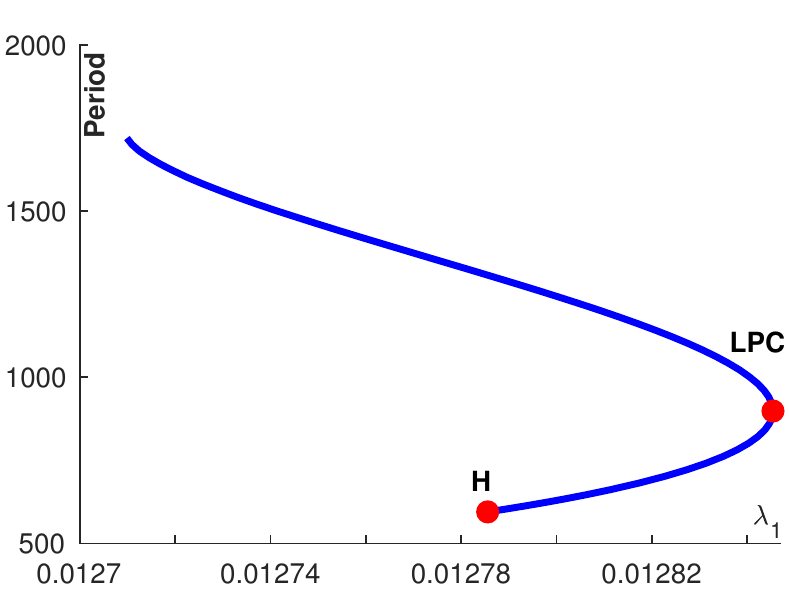} 
    }}
    \quad
    \subfloat[Parameter $\lambda_1$ versus period $T$ of the cycles.]{{
    \includegraphics[scale=0.75]{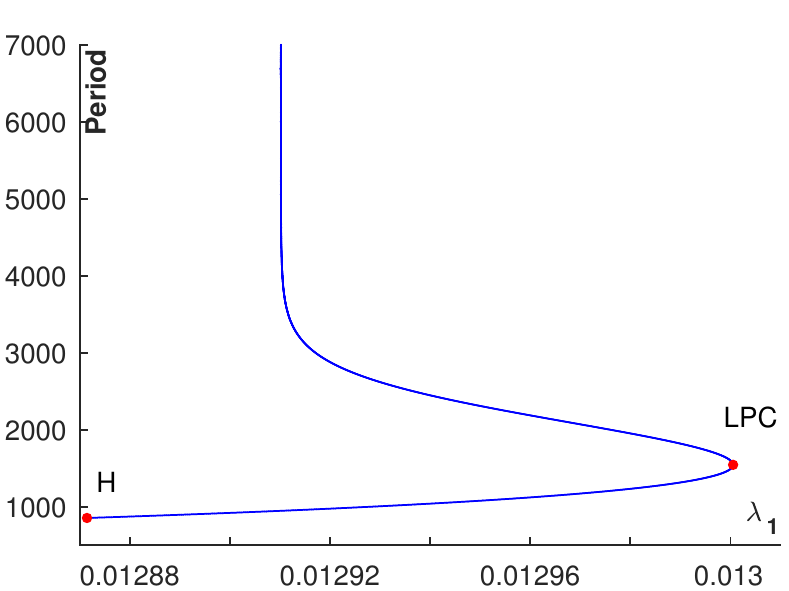} 
    }}\\
    
    \subfloat[Two limit cycles for the same $\lambda_1$ parameter value.]{{
    \includegraphics[scale=0.15]{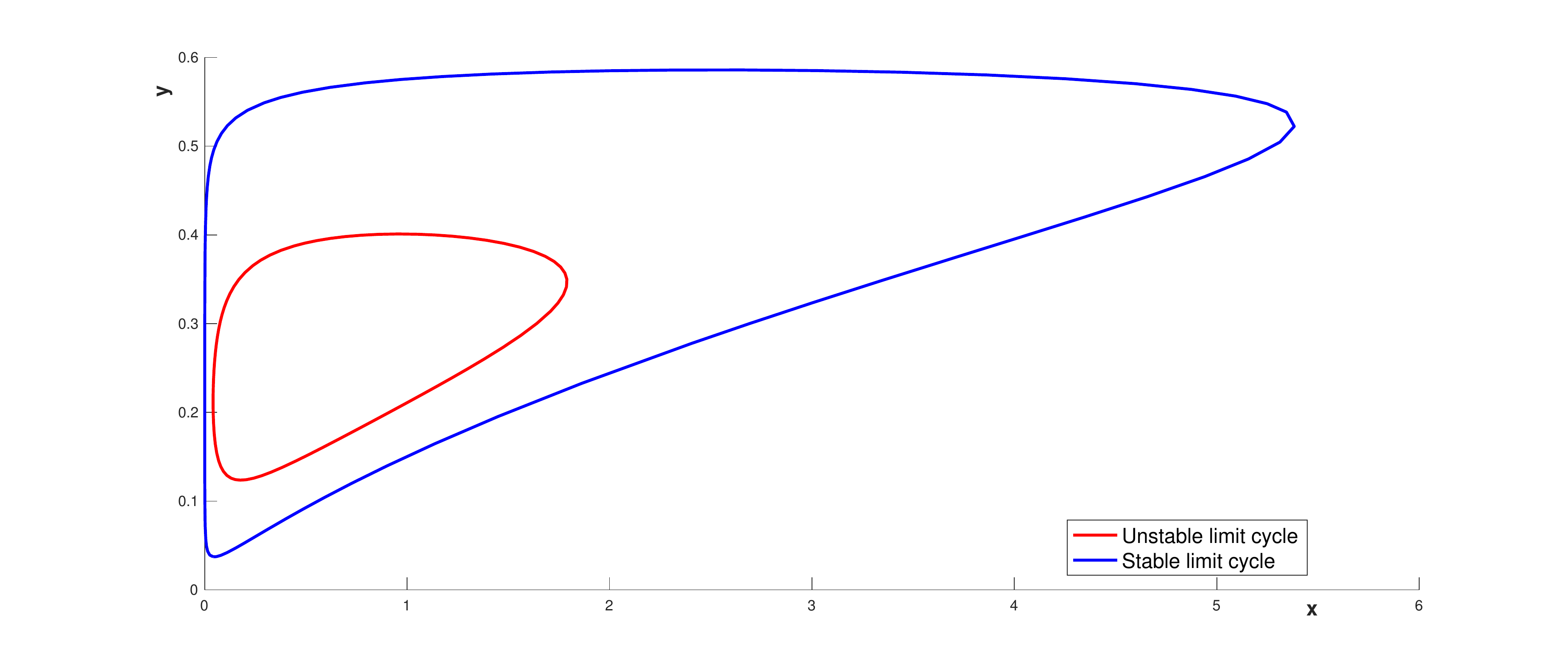} 
    }}
    \quad
    \subfloat[Two limit cycles for the same $\lambda_1$ parameter value.]{{
    \includegraphics[scale=0.25]{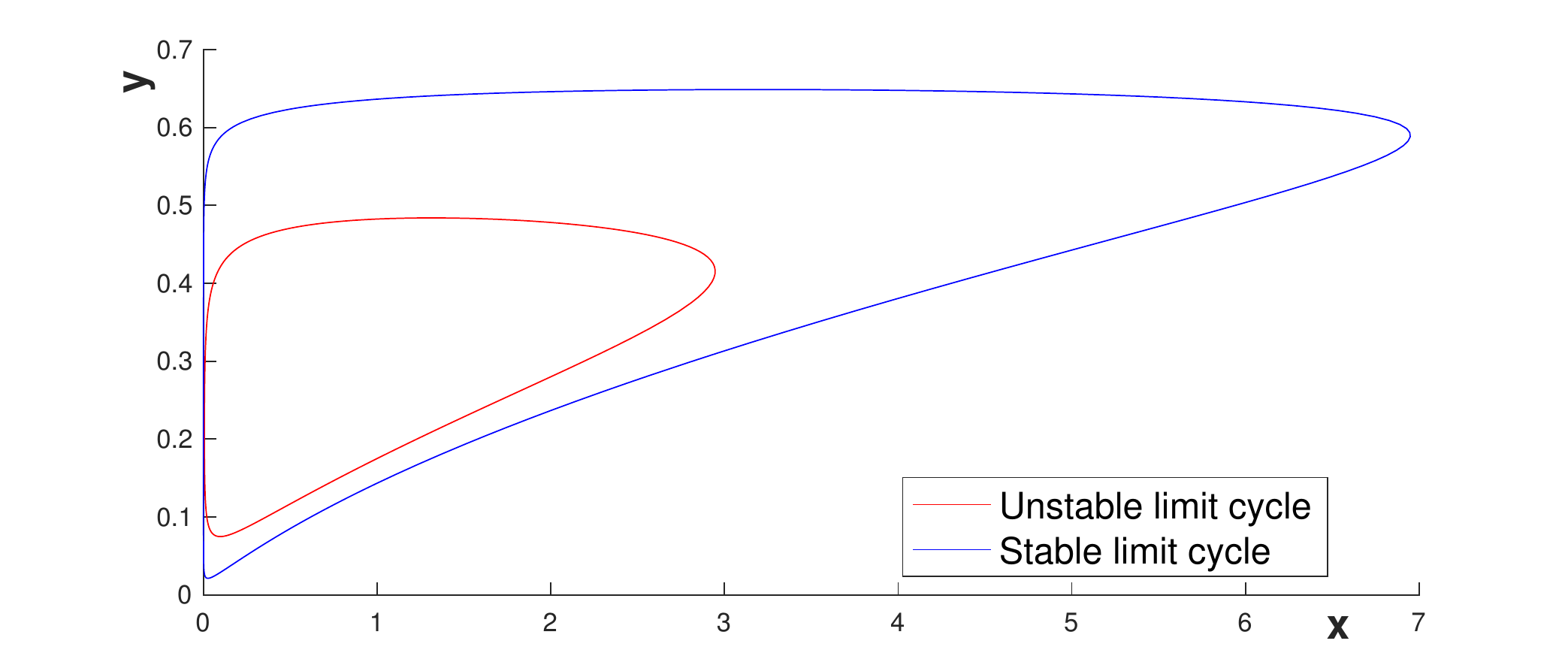} 
    }}
    
    \caption{Coexistence of two limit cycles along the line $C'T'$:  (a), (c)  and (b).   Along the line $C''T''$: (b), (d)  and (f).\label{TwoLC1}}
    \label{Coexistence}
\end{figure}

\begin{figure}
    \centering
    \includegraphics[scale=.5]{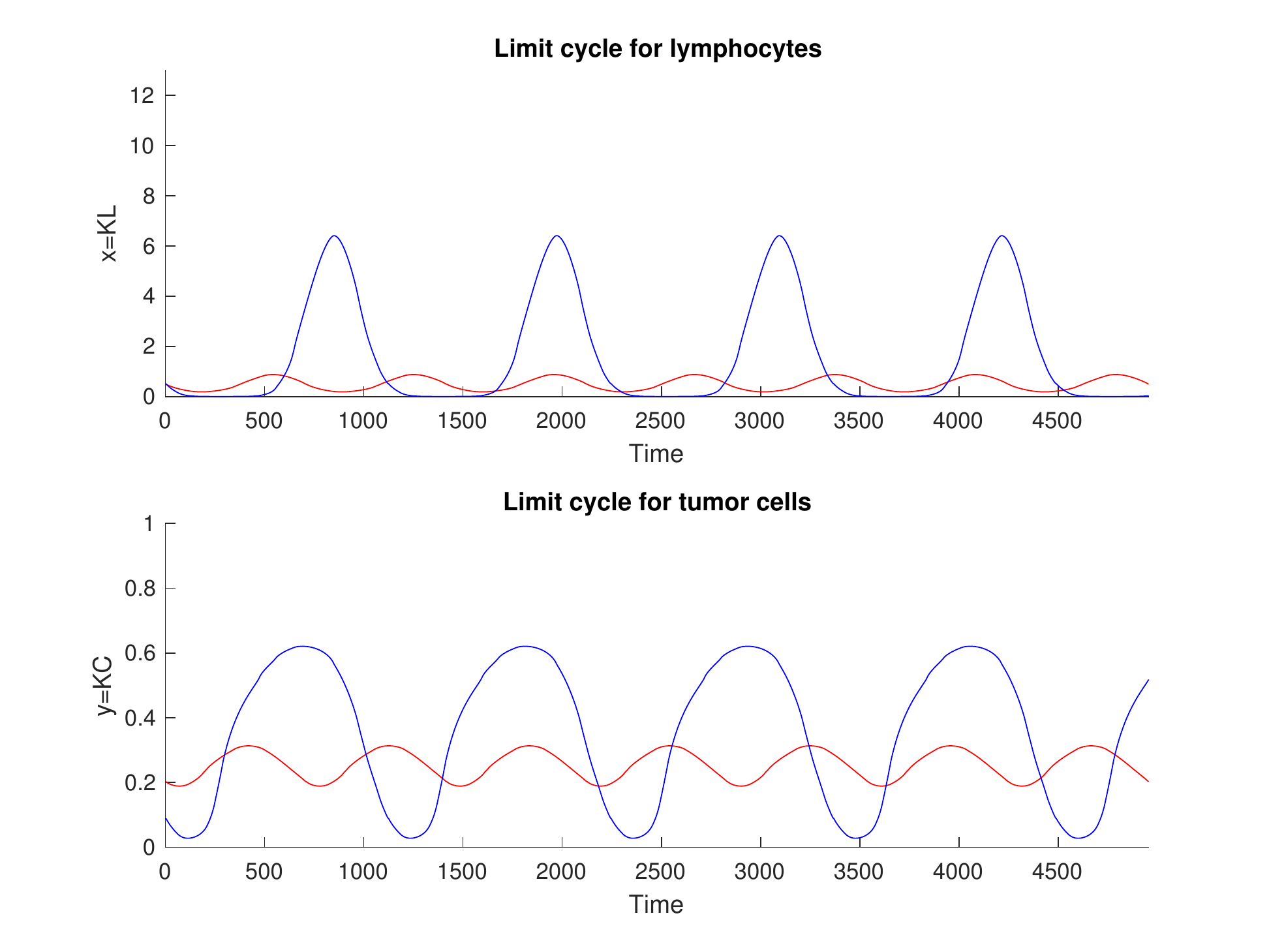}
    \caption{Graphs of coexisting limit cycles of Figure \ref{TwoLC1}. Stable in blue, unstable in red. }
    \label{TwoLC1-bis}
\end{figure}

\section{Implications of the model on the equilibrium phase of immunoediting}
In what follows we will be interested on non-negative values of the parameters and
$x$ within the range $0<x<x_c$. Since $x'<0$ if $x=x_c$ and $y'<0$ if $y=0$, it follows
that the region $0<x<x_c$, $0<y$ is invariant. This delimites the region of real interest (ROI) 
in the model.

\begin{prop}[Elimination threshold]
Given $\alpha_1$, $\alpha_2$, $\lambda_2$, $x_c$, there exists $\lambda_1^*$ such that
if $\lambda_1 >\lambda_1^*$, there exists a curve $y= h(x)$ such that for any initial condition
$(x_0,y_0)$ such that $y_0<h(x_0)$ then there exists $T>0$ such that $y(T)=0$.
\end{prop}
\begin{proof}
Fix  $\alpha_1$, $\alpha_2$, $\lambda_2$ and $x_c$. Since the saddle--node curve is the hyperbola
$\lambda_1\lambda_2^2=const$ (see Proposition 1), then for $\lambda_1$ large enough the unique critical point is the origin
and is a saddle with the positive $y$ axis as a branch of the unstable manifold. Let us consider the rectangular region within the ROI
$$
R=\{(x,y)\mid 0<x<x_c,\quad 0<y<k\}
$$
We have seen that on the boundary $x=x_c$, $x'<0$; on the boundary
$y=0$, $y'<0$. On the upper boundary  $y=k$.
Since $x$ remains bounded, it follows that
$y'=\lambda_2y(1+x)-\alpha_2x$ is positive for   $y=k$  large enough. We now
follow the unstable manifold $W^s(0,0)$ backwards in time. 
A strightforward computation of the stable eigenvalue shows that
a small components of $W^s(0,0)$ belongs to $R$, since there are no
critical points within $R$ it follows that it must intersect the line
$x=x_c$. It remains to show that in fact the component of $W^s(0,0)$
within the region $0<x<x_c$ can be expressed as the graph of
a function $y=h(x)$. Now from the first equation 
$x'=-\lambda_1x(1+x)+\alpha_1 x(1-x/x_c)y^2$, since $x$, $y$ remain bounded and
$\lambda_1$ is large enough, it follows that $x'<0$, and the result
follows.

\end{proof}

The above theorem defines a threshold value of the population of cancer cells $y_c$
given by the intersection of $W^s(0,0)$ and the line $x=x_c$, namely $y_c=h(x_c)$: let $y_0$ be an initial population of cancer cells $y_0<y_c$, for a given growth parameter $\lambda_2$ and interaction constants $\alpha_{1,2}$ then there exists $x_0=h^{-1}(y_0)$ such that for $x_0'>x_0$ the evolution of cancer cells $y(t)$ with initial condition $(x_0',y_0)$  becomes zero. Geometrically, the horizontal line $y=y_0$ in phase space intersects the graph of the curve $y=h(x)$ at a a point $(\bar{x}_0,\bar{y}_0)$ and for an initial population of lymphocites large engouh $x_0<x_0'<x_c$, the solution with inital condition $(x_0',y_0)$ crosses the line $y=0$ for some finite time $T$ and $y(T)=0$. See Figure~\ref{Threshold}

 Notice  that the above dynamics occurrs in the scaled variables $(x,y)$, The branch of the stable manifold $y=h(x)$ transforms back to the original variables $(x,\bar{y})$ into a curve $\bar{y}= h(x)^{3}$ however,  in the original variables the locus $\bar{y}=0$ does not make sense for two
reasons: the first one is that the model breaks down because of the hypothesis of
a spherical tumor. The second is that the system (\ref{Delisi0}) is not Lipschitz for $\bar{y}=0$.
Indeed one expects non--uniqueness as in the well known example $\bar{y}'=\bar{y}^{2/3}$. 
Nevertheless the threshold curve is still defined in the original variables $(x,\bar{y})$, and 
since the change or variables is $C^1$ outside this singular locus $\bar{y}=0$, the same dynamical behaviour occurs in the non-scaled variables.

\begin{figure}[ht]
    \centering
    \includegraphics[height=2in]{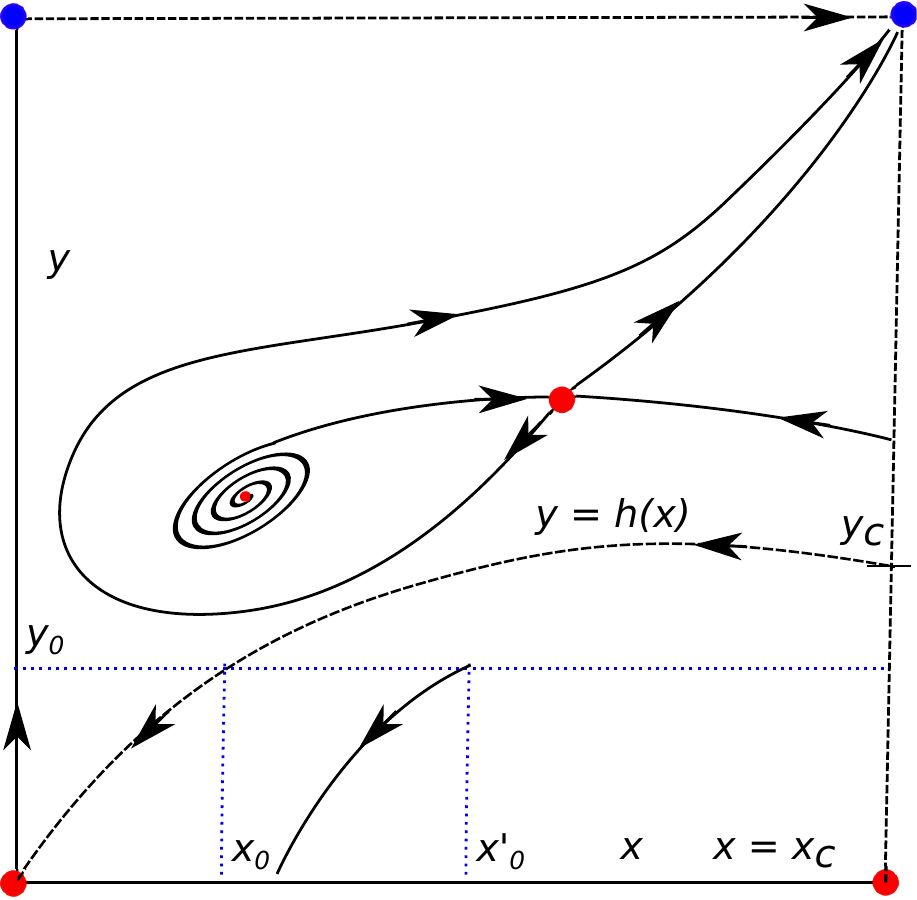}
    \caption{Ilustration of Threshold Theorem}
    \label{Threshold}
\end{figure}

\bigskip
According to the immune edition theory the relation between tumor cells and the immune system is made up of three phases (commonly known as the three E's of cancer): elimination, equilibrium and escape~\cite{Dun}.
Not in these terminology though, Delisi and Resigno \cite{Delisi},   describe these phases in terms of regions delimited by the zeroclines. For example the authors mention that within the region $x'<0$, $y'>0$ denote by $A$ in \cite{Delisi} solution evolves eventualy to escape to $x=x_c$, $y=\infty$. According to the Threshold Theorem~\ref{Threshold}, this is true for initial conditions above the curve $y=h(x)$.  
Here we describe in more detail the three phases according to the regions delimited by the invariant manifold  and basins of attraction. For example, the elimination phase is described as the region below the threshold curve; 
the explosive phase as the basin of attraction of the point at infinity obtained by the compactification of phase space along the $y$ direction (see Appendix~\ref{Blow-up}). The equilibrium phase are the basins of attraction of either a stable anti--saddle or a stable limit cycle. 

The existence of a Bautin bifurcation and the global bifurcation diagram continued numerically, implies the existence of a  triangular region in the plane of paramters $\lambda_1$--$\lambda_2$, for fixed values of $\alpha_1$, $\alpha_2$ and $x_c$ as shown in Figure~\ref{fig:diagram}. Within this region two limit cycles exist and the detailed analysis of the phase diagrams along the lines $CT$, $C'T'$ and $C''T''$ in Figure \ref{schema} and explained in the text, leads to the conclusion that the inner limit cycle is unstable and the exterior one is stable. These two limit cycles are shown in Figure~\ref{TwoLC1}, the correspondig plots agains the time are shown in Figure~\ref{TwoLC1-bis}. This implies that for an initial condition within the interior of the inner limit cycle, the solution tends asymptotically to the values of the stable equlibrium. This would correspond to the  equilibrium phase in the immunoedition theory.  Meanwhile for an initial condition just outside the unstable inner cycle, the population of cancer cells and lymphocytes  grow in amplitud and tends towards a periodic state but of larger amplitude. This yields a new type of qualitative behaviour predicted by the model. 

Escape phase in the immunedition theory corresponds to the basin of attraction of the point at infinity $x=x_c$, $y=+\infty$. The analysis in Appendix B shows that his point is stable, so there is an open set of  initial conditions leading to the escape phase. The basin of atraction of the point at infinity is delimited first by the threshold curve, and secondly by the unstable manifolds of the saddle point with positive coordinates here denoted as $(x_s,y_s)$. The structure of its stable and unstable branches delimits three types of behavior leading to escape. In the first one, for an initial condition $x_0>x_s$ and $y_0$ large enough,  there is a transitory evolution of diminishing values of cancer cells  $x(t)$ less that $x_s$ but finally leading to escape. This region is delimited by the unstable branch connecting $(x_s,y_s)$ and the point at infinity and the stable branch crossing the line $x=x_c$.  The second type of evolution leading to escape occurs for an intial condition of large values of initital population of lymphocites $x_0$ with a great diminishing of $x(t)$, namely less than $x_a$, the abscisa of the anti--saddle critical point $(x_a,y_a)$, following an increse of cancer cells and lymphocites leading finally to escape. This kind of solutions can be described as a turn around the anti--saddle before escaping. A third and more complex behaviour  occurs when the initial condition lies on the boundary of the basin of attraction of a limit cycle. In this situation a small perturbation can lead to oscilations of increasing magnitude and finally to escape.

\appendix

\section{Computation of the first Lyapunov exponent}
In this section we present the main procedure to compute of the first Lyapunov exponent at a Hopf point. 

Let $(x_0,y_0)$ be a critical point. Replacing $x=x_{0}+x_{1}$, $y=y_{0}+y_{1}$ in  (\ref{Delisi}), 
$$
\begin{array}{lcll}
      \frac{dx_{1}}{dt}&=&-\lambda_{1} (x_{0}+x_{1}) (1+x_{0}+x_{1}) +\alpha_{1} \left(1-\frac{x_{0}+x_{1}}{x_{c}}\right)(x_{0}+x_{1})(y_{0}+y_{1})^{2}&\\
    \frac{dy_{1}}{dt}&=&\lambda_{2}(y_{0}+y_{1})(1+x_{0}+x_{1})-\alpha_{2}(x_{0}+x_{1}),
\end{array}
$$
and expanding  we have
\begin{equation}\label{Delisi2}
\begin{array}{lcll}
      x'_{1}&=&a_{0}+a_{1}x_{1}+a_{2}y_{1}+a_{3}x_{1}^{2}+a_{4}y_{1}^{2}+a_{5}x_{1}y_{1}+a_{6}x_{1}y_{1}^{2}+a_{7}x_{1}^{2}y_{1}+a_{8}x_{1}^{2}y_{1}^{2}&\\
    y'_{1}&=&b_{0}+b_{1}x_{1}+b_{2}y_{1}+b_{3}x_{1}y_{1},
\end{array}
\end{equation}
where
\begin{eqnarray*}
a_{0}&=&-\lambda_{1}x_{0}(1+x_{0})+\alpha_{1}\left(1-\frac{x_{0}}{x_{c}}\right)x_{0}y_{0}^{2},\\
b_{0}&=&\lambda_{2}y_{0}(1+x_{0})-\alpha_{2}x_{0}.
\end{eqnarray*}
Of course $a_0=b_0$ yields the equations for the critical points.
The rest of the coefficients are
$$\begin{array}{rclrcl}
a_{1}&=&-\lambda (1+2x_{0})+\alpha_{1}\left(1-\frac{2x_{0}}{x_{c}}\right)y_{0}^{2},&
a_{2}&=&2\alpha_{1}x_{0}y_{0}\left(1-\frac{x_{0}}{x_{c}}\right),\\
a_{3}&=&-\lambda_{1}-\frac{\alpha_{1}y_{0}^{2}}{x_{c}},&
a_{4}&=&\alpha_{1} x_{0}\left(1-\frac{x_{0}}{x_{c}}\right),\\
a_{5}&=&2\alpha_{1}y_{0}\left(1-\frac{2x_{0}}{x_{c}}\right),&
a_{6}&=&\alpha_{1}\left(1-\frac{2x_{0}}{x_{c}}\right),\\
a_{7}&=&-2\frac{\alpha_{1}y_{0}}{x_{c}},&
a_{8}&=&-\frac{\alpha_{1}}{x_{c}},\\
b_{1}&=&\lambda_{2}y_{0}-\alpha_{2},&
b_{2}&=&\lambda_{2}(1+x_{0}),\\
b_{3}&=&\lambda_{2}.
\end{array}
$$

Consider the linear part $x'=Ax$ where $x=(x_1,y_1)^T$ and
$$A=  \begin{pmatrix}
 a_{1} & a_{2} \\ 
 b_{1} & b_{2}
 \end{pmatrix}
$$
Perform the linear  change of coordinates
\begin{equation}\label{Cambio}
    Y=Mx,
\end{equation}
where $Y=(Y_1,Y_2)^T$ and
$$
M=
\begin{pmatrix}
b_{2} & -a_{2}\\
a_{1}b_{2}-a_{2}b_{1} & 0
\end{pmatrix}$$ 
then the linear system is transformed into
$$
Y'=RY,\quad R=\begin{pmatrix} 0 & 1 \\ -\det(A) & \tr(A)
\end{pmatrix}
$$
Then $R$ has the canonical form
$$R=
\begin{pmatrix}
0 & 1\\
-\omega^2 & 2\mu
\end{pmatrix}
$$
where we have supposed  and set that $0<\det(A)\equiv \omega^2$, $\tr(A)=2\mu$ and $\mu^2-\omega^2<0$ so we have complex eigenvalues 
$\lambda=\mu \pm i \sqrt{\omega^{2}-\mu^{2}}$,

Let us consider that the real part of the eigenvalues is zero ($\mu=0$), then 
$$
R_{0}=\begin{pmatrix}
0 & 1\\
-\omega^{2} & 0
\end{pmatrix},
$$
and we want to find vectors $q_{0}$ y $p_{0}$, such that $R_{0}q_{0}=i\omega q_{0}$, $R_0^{T}p_{0}=-i\omega p_{0}$, $\langle p_{0},q_{0}\rangle =1$ and $ \langle p_{0},\bar{q}_{0}\rangle=0$.
We find
$$
q_{0}=\frac{1}{2 i \omega}\begin{pmatrix}
1\\
i \omega
\end{pmatrix}
$$
and 
$$p_{0}=\begin{pmatrix}
-i\omega\\
1
\end{pmatrix}.
$$

Let us trasform the complete system (\ref{Delisi2})  at a critical point with complex eigenvalues $\lambda\pm i \omega$
$$
x'=Ax+H_{2}(x)+H_{3}(x)+\cdots,
$$
by means of the change of variables
(\ref{Cambio}) then

\begin{eqnarray*}
Y'&=&Mx'\\
&=&MA_{0}x+MH_{2}(x)+MH_{3}(x)+\cdots,\\
&=&MA_{0}\left(M^{-1}Y\right)+MH_{2}\left(M^{-1}Y\right)+MH_{3}\left(M^{-1}Y\right)+\cdots \\
&=&
R\begin{pmatrix}
Y_{1}\\
Y_{2}
\end{pmatrix}+ K_{2}\begin{pmatrix}
Y_{1}\\
Y_{2}
\end{pmatrix}+K_{3}\begin{pmatrix}
Y_{1}\\
Y_{2}
\end{pmatrix}+\cdots
\end{eqnarray*}
where $K_l = MH_lM^{-1}$, for $l=1,2,\ldots$

Now introduce the complex variable $z$ by
$$\begin{pmatrix}
Y_{1}\\
Y_{2}
\end{pmatrix}=z q_{0}+\bar{z}\bar{q_{0}},$$
then system is reduced to the normal form
\begin{eqnarray*}
z'&=&\lambda z +\langle p_{0},K_{2}(z q_{0}+\bar{z}\bar{q_{0}})\rangle +\cdots \\
&=&\lambda z+ G_{2}(z,\bar{z})+G_{3}(z,\bar{z})+\cdots \\
&=& \lambda z+ \frac{g_{20}}{2}z^2 +g_{11}z\bar{z}+ \frac{g_{02}}{2}\bar{z}^2+\cdots
\end{eqnarray*}
where 
$$
G_{l}= \langle p_{0},K_{l}(z q_{0}+\bar{z}\bar{q_{0}})\rangle ,\quad l=2,3\ldots
$$
and $g_{ij} = \frac{1}{i!j!}\frac{\partial G_l}{\partial z^{i}\bar{z}^{j}}$, for $i,j=0,1,\ldots$ We will need the expansion up to third order terms, in particular  the coefficient $g_{11}$ at the third order

We will compute the first Lyapunov coefficients using the formulas (3.18) in \cite{Yuri} for the coefficient $c_1(0)$ of the Poincaré normal form and
\begin{equation}\label{l1}
\ell_1(0)= \frac{\mbox{Re}(c_1(0))}{\omega}
\end{equation}
where
$$
c_1(0)=\frac{g_{21}}{2}
+\frac{g_{20} g_{11}i\omega}{2\omega^2}-i\frac{g_{11}\bar{g_{11}}}{\omega} -i\frac{g_{02}\bar{g_{02}}}{6\omega}
$$

Observe that the change of coordinates (\ref{Cambio}) contains the coordinates of the critical point $(x_0,y_0)$ and so the coefficients $g_{ij}$. Therefore, we have to impose on  the formal expression we get using (\ref{l1}) from the coefficients $g_{ij}$ up to third order,  the restriction of a critical point, with zero real part and positive determinant equal to $\omega^2$. We achieve this as follows:
The expression (\ref{l1}) is a polynomial expression  depending on $(x_0,y_0)$ and the parameters $P(x_0,y_0,\lambda_1,\lambda_2,x_c,\alpha_1,\alpha_2)$.  Firstly we eliminate $y_0$ using (\ref{y0}) obtaining a polynomial expression in $x_0$ of order 19 and the parameters  and  still denote by $P(x_0,\lambda_1,\lambda_2,x_c,\alpha_1,\alpha_2)$. The abscisa $x_0$ of the critical point satisfy the cubic equation (\ref{x0}) written here as $Q(x_0,\lambda_,\psi,x_c)$.
Suprisingly, the coefficients of $P$ and $Q$ can be expressed solely in terms of the combination of parameters $\lambda$, $\psi$ and $x_c$. Next
we  eliminate $x_0$ using the resultant
$$
R_1(\lambda_,\psi,x_c)=\Res(P(x_0,\lambda_,\psi,x_c),Q(x_0,\lambda_,\psi,x_c),x_0).
$$
Also the Hopf surface can be expressed in terms of the same combination of parameters as shown in (\ref{Hopf}) as $R_2(\lambda,\psi,x_c)=0$, then we compute
$$
R_3(\lambda,x_c)= \Res(R_1(\lambda,\psi,x_c),R_2(\lambda,\psi,x_c),\psi)
$$ 
and we get from a  non trivial factor of $R_3$
\begin{equation}\label{bautin}
    \lambda=\frac{-3+x_c}{3(1+x_c)}
\end{equation}

Finally, substituting (\ref{bautin}) in the Hopf surface $R_2(\lambda_,\psi,x_c)=0$ we get the nonnegative solution
$$
\psi=\frac{\sqrt{x_c}\left(
(-27+x_c)\sqrt{x_c}+ (9+x_c)^{3/2}
\right)}{27(1+x_c)^2}.
$$

\section{Blow up of infinity\label{Blow-up}}
In order to study solutions that escape to infinity in the direction $y\to\infty$ we perform a blow up of infinity 
by the change of variables $(x,y)\mapsto (x,v=x/y)$, a further rescaling of time $dt/dt'=v^2$ extends  the system up to $v=0$ corresponding to infinity $y=\infty$, $x>0$ (\ref{Delisi}) 
\begin{eqnarray}
\frac{dx}{dt'}&=&-\lambda_1 x(1+x) v^2+\alpha_1 x^3\left(1-\frac{x}{x_c}\right),\nonumber\\
\frac{dy}{dt'} &=& vx\left(\alpha_1 x\left(1-\frac{x}{x_c}\right)-\lambda_2 \right)+\alpha_2 v^2 -v^3(\lambda_1(1+x)+\lambda_2) .\label{infinity}
\end{eqnarray}
We see that $v=0$ becomes invariant and the reduced system at infinity is
$$
\frac{dx}{dt'}= x^3(1-x)
$$
showing that  along $v=0$, $x>0$,  $x=x_c$ an attractor.

To determine the local phase portrait  of system (\ref{infinity}) at the critical point $x=x_c,v=0$, we compute  its linearization 
$$
A=\begin{pmatrix}
- x_c^2 \alpha_1 & 0\\ 0 & - x_c\lambda_2
\end{pmatrix}
$$
thus $(x_c,v=0)$ is an attractor. The origin $x=0=v$ is also a degenerate critical point with zero linear part with terms of third order the least. Performing a radial blow using polar coordinates $x=r\cos{\theta}$, $v=r\sin{\theta}$
we get
\begin{eqnarray*}
\frac{dr}{dt}&=& r(-\lambda_1+\alpha_1\cot^2{\theta}-\lambda_2\sin^2{\theta})+\\
&& r^2\left(-\lambda_1\cos^2{\theta}+\alpha_2\sin^3{\theta}-\frac{\alpha_1}{x_c}\cot^2{\theta}-(\lambda_1+\lambda_2)\cos{\theta}\sin^2{\theta}\right)\\
\frac{d\theta}{dt}&=&-\lambda_2\cos{\theta}\sin{\theta}- r\cos{\theta}\sin{\theta}(\lambda_2\cos{\theta}-\alpha_2\sin{\theta})
\end{eqnarray*}
which shows that $r=0$, $0<\theta<\pi/2$ is invariant. Setting $r=0$ we get
$$
\frac{d\theta}{dt}=-\lambda_2\cos{\theta}\sin{\theta}
$$
which is always negative for $0<\theta<\pi/2$. Thus the origin is a degenerate critical point with a hyperbolic sector.

\section{Numerical continuation}
Following \cite{Dan} we take the  numerical values
$$
\lambda_1 = 0.01, \, \lambda_2 = 0.006672, \, \alpha_1 = 0.297312, \,\alpha_2 = 0.00318,\, xc = 2500
$$
satisfying conditions (\ref{Delisi}) for a BT bifurcation, and the coordinates $x_0=1.9976$, $y_0=0.317619$ for the critical point, according to (\ref{y0}), (\ref{x0}). 
 Figure~\ref{fig:diagram}~(a)--(b) shows the family of homoclinic connections in phase space,  originating from  the BT critical point. Since continuing the family of homoclinics from the BT point sometimes is difficult (see \cite{Hdaibat2}),  for the computation of the initial member of the family of homoclinics we use the homotophy method near the previous values of $\lambda_1$, $\lambda_2$ and then continue forward and backward to assure that the family originates from the BT point. The curve of homoclinics is shown in Figure~\ref{fig:diagram} as the violet curve. The Bautin point (GH) is detected by continuing the Hopf curve from the BT point.

\begin{figure}[ht]
    \centering
    \subfloat[Continuation of the homoclinic orbit of avascular Delisi model (\ref{Delisi}) near Bogdanov-Takens bifurcation.]{
    \includegraphics[height=1.7in]{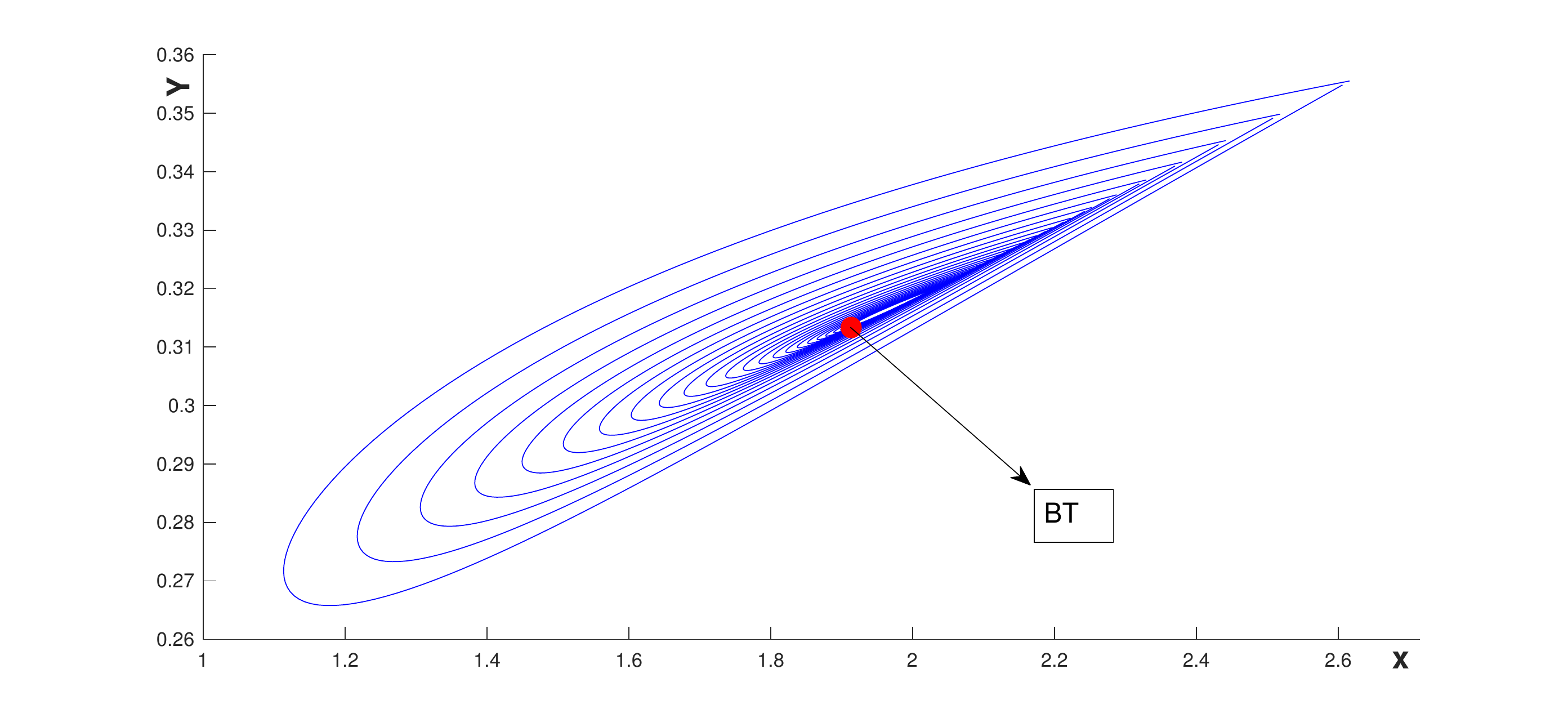}
    }\\ 
 \subfloat[Bifurcations  of avascular Delisi model (\ref{Delisi}) from parameters values $\alpha_{1}=0.297312$, $\alpha_{2}=0.00318$ and $x_c=2500$.]{
    \includegraphics[height=2in]{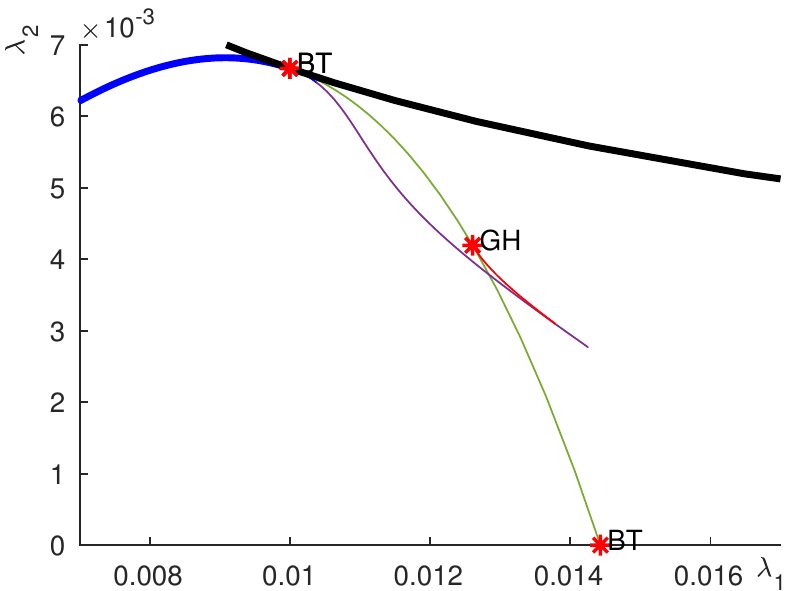}
    }
\caption{Numerical continuation of bifurcation diagram with MatCont. Saddle-node: black; Hopf: green; limit point of cycles: red; symmetric saddles:blue; homoclinic: violet.}
\label{fig:diagram}    
\end{figure}

The Delisi model diagram bifurcation is shown in  Figure \ref{fig:diagram}~(b). The saddle-node bifurcation curve is shown in black, the green corresponds to the Hopf bifurcation, the curve in red corresponds to the saddle-node bifurcation of periodic orbits (limit point of cycles) and the blue one to  symmetric saddles.

\end{document}